\documentclass[runningheads]{llncs}
\usepackage{bbm}
\usepackage{times}
\usepackage{verbatim}
\usepackage{mathtools}
\usepackage{graphicx}
\usepackage{color}
\usepackage{amsfonts}
\newcommand{\komp}      {{^\prime}}
\newcommand{\nega}      [1] {{#1}\komp}
\newcommand{\te}{{\mathbin{*\mkern-9mu \circ}}}
\newcommand{\ite}[1]{\mathbin{\rightarrow_{#1}}}
\newcommand{\g}                 [2] {{#1} \mathbin{\te} {#2}}
\newcommand{\res}               [3] {{#2}\mathbin{\ite{#1}}{#3}}
\newcommand{\kompM}[1]{^{{\prime^{\mkern-4.5mu^{_{#1}}}}}}
\newcommand{\ted}{\mathbin{\diamond}}
\newcommand{\negaM}      [2] {{#2}\kompM{#1}}

\newcommand{\lex}                                        {\overset{\leftarrow}{\times}}
\newcommand{\plexII}               {\overset{\Twoheadleftarrow}{\times}}
\newcommand{\plexI}               {\overset{\threeheadleftarrow}{\times}}
\newcommand{\PLPI}            [3] {{#1}_{{#2}}\plexI{{#3}}}
\newcommand{\PLPII}            [2] {{#1}\plexII{{#2}}}
\newcommand{\PLPIII}            [4] {{#1}_{{#2}_{#3}}\plexI{{#4}}}
\newcommand{\PLPIV}            [3] {{#1}_{#2}\plexII{{#3}}}
\newcommand{\Twoheadleftarrow}               {\leftarrow \mkern-12.25mu \leftharpoonup}
\newcommand{\threeheadleftarrow}               {\leftarrow \mkern-12.5mu \leftarrow}

\begin{document}
\title{Group-like Uninorms\thanks{The present scientific contribution was supported by the GINOP 2.3.2-15-2016-00022 grant
and the Higher Education Institutional Excellence Programme 20765-3/2018/FEKUTSTRAT of the Ministry of Human Capacities in Hungary.
}}
%
%
\author{S\'andor Jenei\orcidID{0000-0001-8664-4226}}
\authorrunning{S. Jenei}
%
\institute{University of P\'ecs, P\'ecs, Hungary
\email{jenei@ttk.pte.hu}\\
\url{http://jenei.ttk.pte.hu/home.html}
}
\maketitle              
\begin{abstract}
Uninorms play a prominent role both in the theory and the applications of Aggregations and Fuzzy Logic. In this paper the class of group-like uninorms is introduced and characterized.
Other, more algebraic terminologies for this class of algebras are odd involutive commutative residuated lattices over $[0,1]$ or odd involutive FL$_e$-algebras over $[0,1]$.
First, two variants of a general construction -- called partial-lexicographic product --  will be recalled from \cite{Jenei_Hahn}; these  construct odd involutive FL$_e$-algebras.
Then two particular ways of applying the partial-lexicographic product construction will be specified.
The first method constructs, starting from $\mathbb R$ (the additive group of the reals) and modifying it in some way by $\mathbb Z$'s (the additive group of the integers), what we call basic group-like uninorms,
whereas with the second method one can modify any group-like uninorm by a basic group-like uninorm to obtain another group-like uninorm.
All group-like uninorms obtained this way have finitely many idempotent elements.
On the other hand, we prove that given any group-like uninorm which has finitely many idempotent elements, it can be constructed by consecutive applications of the second construction (finitely many times) using only basic group-like uninorms as building blocks.
Hence any basic group-like uninorm can be built using the first method, and any group-like uninorm which has finitely many idempotent elements can be built using the second method from only basic group-like uninorms.
In this way a complete characterization for group-like uninorms which possess finitely many idempotent elements is given: ultimately, all such uninorms can be built from $\mathbb R$ and $\mathbb Z$.
This characterization provides, for potential applications in several fields of fuzzy theory or aggregation theory,  the whole spectrum of choice of those group-like uninorms which possess finitely many idempotent elements.

\end{abstract}
%
%
 
\section{Introduction}
Aggregation operations are crucial in numerous pure and applied fields of mathematics.
Fuzzy Theory is another large field, involving both pure mathematics and an impressive range of applications.
Mathematical fuzzy logics have been introduced in \cite{hajekbook}, and the topic is a rapidly growing field ever since.
In all these fields (and the list in far from being exhaustive) a crucial role is played by t-norms, t-conorms, and uninorms \cite{kmpbook}.

Introduced in \cite{UniIntro}, a {\em uninorm} $U$,  is a function of type $[0,1]\times[0,1]\to[0,1]$, that is, binary operations over the closed real unit interval $[0,1]$, such that the following axioms are satisfied.
$$
\begin{array}{ll}
U(x,y)=U(y,x)							&  \mbox{\ \ \ \ (Symmetry)}\\
\mbox{If $y\leq z$ then $U(x,y)\leq U(x,z)$}	& \mbox{\ \ \ \ (Monotonicity)}\\
U(U(x,y),z)=U(x,U(y,z))						& \mbox{\ \ \ \ (Associativity)}\\
\mbox{There exists $t\in]0,1[$ such that $U(x,t)=x$}	& \mbox{\ \ \ \ (Unit Element)}\\
\end{array}
$$
Establishing the structure theory of the whole class of uninorms seems to be quite difficult. 
Several authors have characterized particular subclasses of them, see e.g., \cite{IdUni,FdB1997,D07,RatUni,FodStruct,MZ2018,MZ2018b,PM2014,SZXX}.
Uninorms are interesting not only for a structural description purpose, but also different generalizations of them play a central role in many studies, see  \cite{A07,D14} for example.
Group-like uninorms, to be introduced below, form a subclass of involutive uninorms, and involutive uninorms play the same role among uninorms as the \L ukasiewicz t-norm or in general the class of rotation-invariant t-norms \cite{rotinvsemigroups,JenRotTNORM,JenRotAnnTNORM,RotRevisited,RotTriple} do in the class of t-norms.
In this paper we shall give a complete characterization for the class of group-like uninorms which have finitely many idempotent elements.

To this end, first a few notions should follow here:
Residuation is a crucial property in Mathematical Fuzzy Logics, and in Substructural Logics, in general \cite{gjko,MetMon 2007}.
A uninorm is residuated if there exists a function $I_U$ of type $[0,1]\times[0,1]\to[0,1]$, that is, a binary operation on $[0,1]$, such that the following is satisfied: $U(x,y)\leq z$ if and only if $I_U(x,z)\geq y$.
Frequently one uses the infix notation for a uninorms, too, and writes $\g{x}{y}$ in stead of $U(x,y)$, and $\res{\te}{x}{y}$ instead of $I_T(x,y)$.
A generalization of residuated  t-norms and uninorms is the notion of  FL$_e$-algebras.
This generalization is done by replacing $[0,1]$ by an arbitrary lattice, possibly without top and bottom elements:
an {\em FL$_e$-algebra}\footnote{Other terminologies for FL$_e$-algebras are: pointed commutative residuated lattices or pointed commutative residuated lattice-ordered monoids.} 
is a structure $( X, \wedge,\vee, \te, \ite{\te}, t, f )$ such that 
$( X, \wedge,\vee )$ is a lattice, $( X,\leq, \te,t)$ is a commutative, 
residuated\footnote{That is, there exists a binary operation $\ite{\te}$ such that $\g{x}{y}\leq z$ if and only if $\res{\te}{x}{z}\geq y$; this equivalence is often called adjointness condition, ($\te,\ite{\te}$) is called an adjoint pair. Equivalently, for any $x,z$, the set $\{v\ | \ \g{x}{v}\leq z\}$ has its greatest element, and $\res{\te}{x}{z}$ is defined as this element: $\res{\te}{x}{z}:=\max\{v\ | \ \g{x}{v}\leq z\}$; this is often referred to as the residuation property.}
monoid
, and $f$ is an arbitrary constant.
One defines $\nega{x}=\res{\te}{x}{f}$ and calls an FL$_e$-algebra {\em involutive} if $\nega{(\nega{x})}=x$ holds.
Call an involutive FL$_e$-algebra {\em group-like} or {\em odd} if $t=f$ holds. 
Throughout the paper we shall simply write \lq\lq\ odd FL$_e$-algebra\rq\rq\ instead of \lq\lq odd involutive FL$_e$-algebra\rq\rq\ for short.
For a odd FL$_e$-algebra $\mathbf X$, let $gr(X)$ be the set of invertible elements of $\mathbf X$. It turns out that there is a subalgebra of $\mathbf X$ on $gr(X)$, denote it by $\mathbf X_\mathbf{gr}$ and call it the group part of $\mathbf X$. \\
Speaking in algebraic terms, t-norms and uninorms are the monoidal operations of commutative linearly ordered monoids over $[0,1]$.
Likewise, residuated t-norms and uninorms are just the monoidal operations of FL$_e$-algebras over $[0,1]$.
According to the terminology above, the class of involutive t-norms constitutes the \L ukasiewicz t-norm, and all rotation-invariant t-norms (aka. IMTL-algebras over $[0,1]$) in general.
Also according to the terminology above, 
\begin{definition}\rm
We call a uninorm $\te$ {\em group-like} if it is residuated, $\nega{\nega{x}}=x$ holds where $\nega{x}=\res{\te}{x}{t}$, and $\nega{t}=t$.
\end{definition}
For group-like uninorms (and also for bounded odd FL$_e$-algebras, in general) we know more about their behaviour in the boundary, 
 as it holds true that 
$$
U(x,y)=
\left\{
\begin{array}{ll}
\in]0,1[						& \mbox{ if $x,y\in]0,1[$}\\
0							& \mbox{ if $\min(x,y)=0$}\\
1							& \mbox{ if $x,y>0$ and $\max(x,y)=1$}
\end{array}
\right. .
$$
Therefore, values of a group-like uninorm $U$ in the open unit square $]0,1[^2$ fully determine $U$. As a consequence, one can view group-like uninorms as binary operations on $]0,1[$, too.
Because of these observations, throughout the paper we shall use the term {\em group-like uninorm} is a slightly different manner: Instead of requiring the underlying universe to be $[0,1]$, we only require that the underlying universe is order isomorphic to the open unit interval $]0,1[$. This way, for example the usual addition of real numbers, that is letting $V(x,y)=x+y$, becomes a group-like uninorm in our terminology. This is witnessed by any order-isomorphism from $]0,1[$ to $\mathbb R$, take for instance $\varphi(x)= \tan(\pi x-\frac{\pi}{2})$.
Using $\varphi$, any group-like uninorm (on $\mathbb R$, for example) can be carried over to $[0,1]$ by letting, in our example,
$$
U(x,y)=
\left\{
\begin{array}{ll}
\varphi^{-1}(V(\varphi(x),\varphi(y)))			& \mbox{ if $x,y\in]0,1[$}\\
0							& \mbox{ if $\min(x,y)=0$}\\
1							& \mbox{ if $x,y\neq 0$ and $\max(x,y)=1$}
\end{array}
\right. .
$$

As said above, odd FL$_e$-chains are involutive FL$_e$-chains satisfying the condition that the unit of the monoidal operation coincides with the constant that defines the order-reversing involution $\komp$; in notation $t=f$.
Since for any involutive FL$_e$-chain $\nega{t}=f$ holds, one extremal situation is the integral case, that is, when $t$ is the top element of the universe and hence $f$ is its bottom one (this is essentially the t-norm case), and the other extremal situation is the group-like case when the two constants coincide \lq\lq and are in the middle\rq\rq.
Prominent examples of odd FL$_e$-algebras are lattice-ordered Abelian groups 
and odd Sugihara monoids, 
the former constitutes an algebraic semantics of Abelian Logic \cite{MeyerAbelian}, and the latter constitutes an algebraic semantics of a logic at the intersection of relevance logic and fuzzy logic \cite{GalRaf}.
These two examples are also extremal in the sense that lattice-ordered Abelian groups have a single idempotent element, namely the unit element, whereas all elements of any odd Sugihara monoid are idempotent.
\begin{figure}
\begin{center}
  \includegraphics[width=0.412\textwidth]{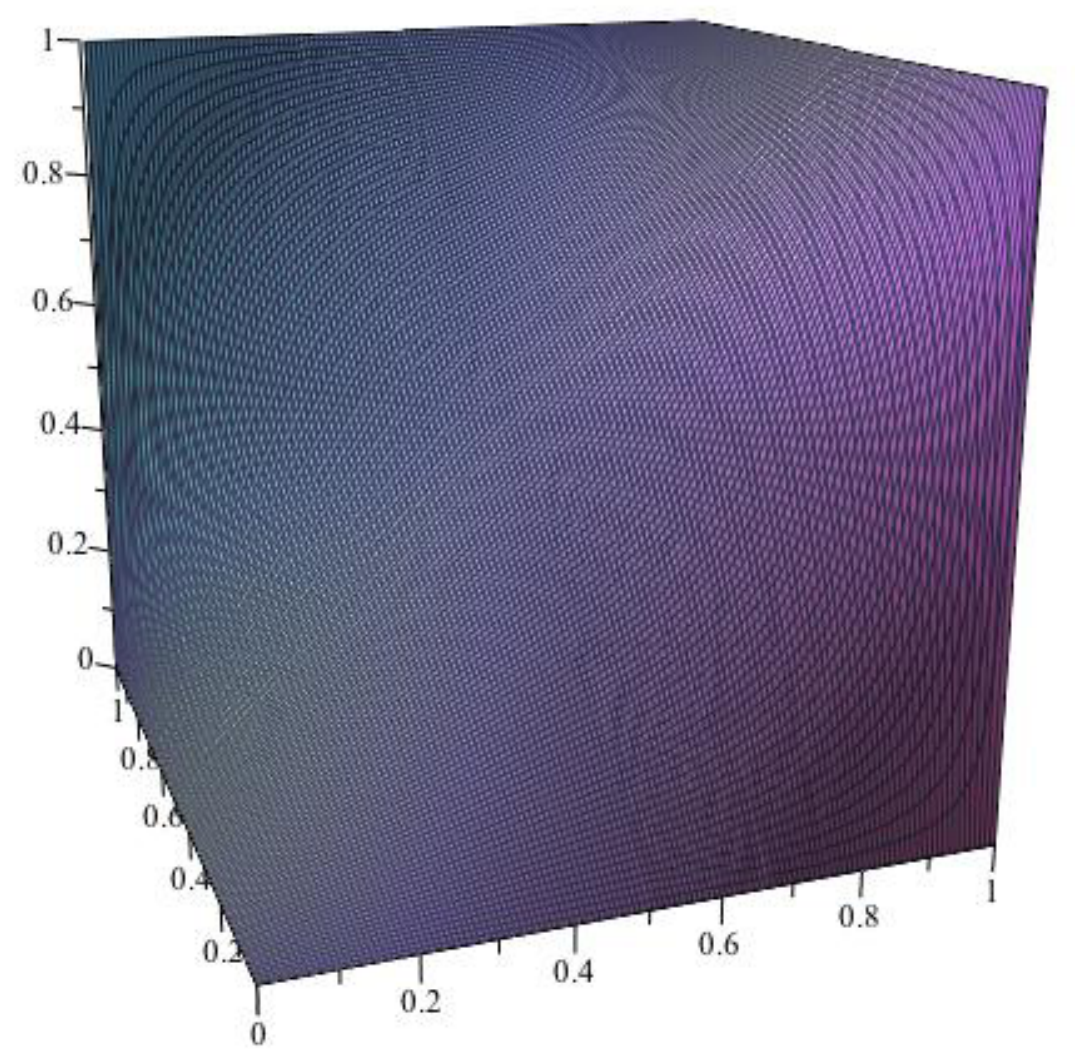} \ \  
  \includegraphics[width=0.41\textwidth]{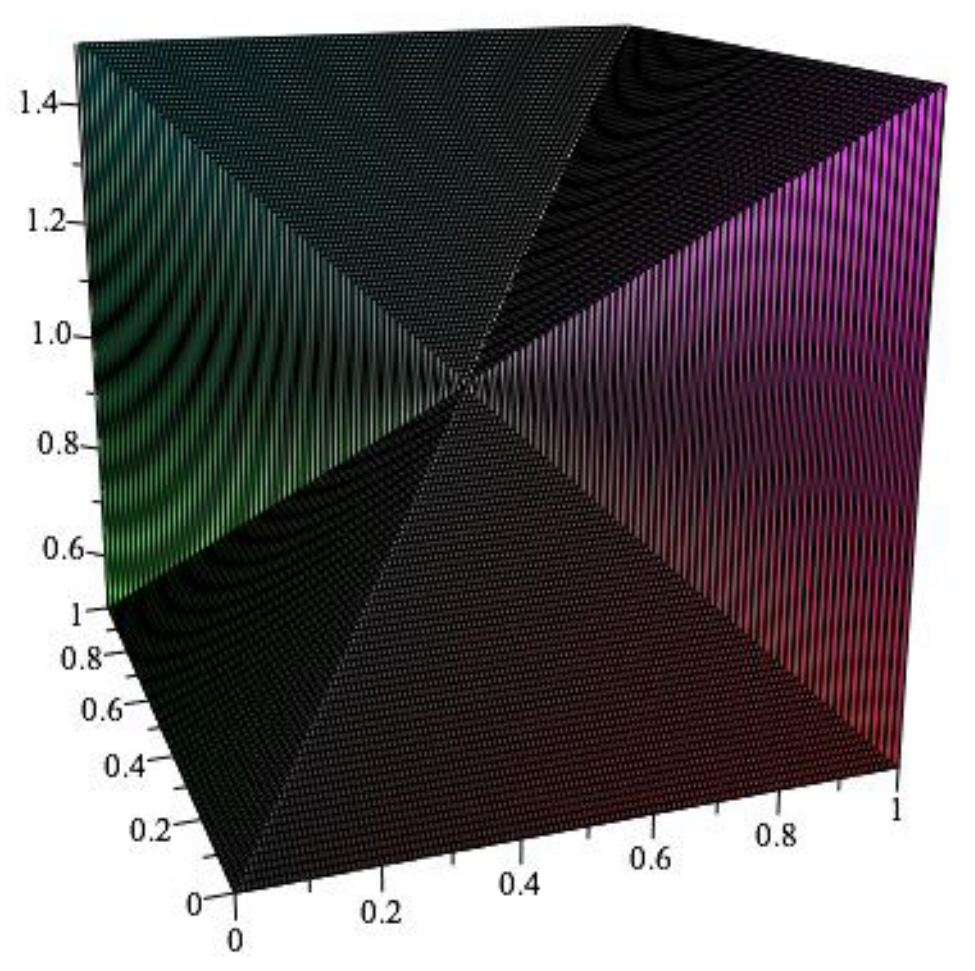} \ \
\caption{Visualization: The only linearly ordered abelian group over $]0,1[$ (left) and the only odd Sugihara monoid over $]0,1[$ (right).}
\label{fig:1}
\end{center}
\end{figure}
In order to narrow the gap between the two extremal classes mentioned above, in \cite{Jenei_Hahn} a deep knowledge have been gained about the structure of odd FL$_e$-chains, including a Hahn-type embedding theorem, and a representation theorem by means of linearly ordered abelian groups and a there-introduced construction, called partial-lexicographic product.
This representation theorem has a crucial role in proving the results of the present paper.
First, we adopt the partial-lexicographic product construction to the setting of group-like uninorms, by introducing two particular ways of applying it. These applications use only  $\mathbb R$ and $\mathbb Z$.
With these variants one can construct group-like uninorms having finitely many idempotent elements.
Our main theorem asserts that all group-like uninorms having finitely many idempotent elements can be constructed by using these two variants. 
Ultimately, it follows that all these uninorms can be constructed by the mentioned two variants of the partial lex-product construction using only $\mathbb R$ and $\mathbb Z$.

Another interpretation of the same result is that all these uninorms can be built by the second variant of the partial lex-product construction from basic group-like uninorms.
If understood this way then there is a striking similarity between this characterization and the well-known ordinal sum representation of continuous t-norms of Mostert and Shields as ordinal sums of continuous archimedean t-norms \cite{Mos57}: 
replace \lq t-norm\rq\, by \lq uninorm\rq, 
\lq continuous\rq\, by \lq group-like with finitely many idempotent elements\rq,
\lq continuous archimedean t-norm\rq\, by \lq basic group-like uninorm\rq,
and
\lq ordinal sum construction\rq\, by \lq the second variant of the partial lex-product construction\rq.
Besides, according to the classification of continuous archimedean t-norms, any continuous archimedean t-norm is order-isomorphic to either the \L ukasiewicz t-norm or the Product t-norm, so there are two prototypes. In our setting basic group-like uninorms have $\aleph_0$ prototypes, one for each natural number.

\section{Preliminaries}

\begin{definition}\label{nyilak}\rm
For a chain (a linearly ordered set) $(X, \leq)$ and for $x\in X$ define the predecessor 
$x_\downarrow$ of $x$ to be the maximal element of the set of elements which are smaller than $x$, if it exists, define $x_\downarrow=x$ otherwise. 
Define the successor $x_\uparrow$ of $x$ dually. 
We say for $Z\subseteq X$ that $Z$ is {\em discretely embedded} into $X$ if 
for $x\in Z$ it holds true that $x\notin\{ x_\uparrow,x_\downarrow\}\subseteq Z$. 
If $\mathbf H$ is subalgebra of an odd FL$_e$-algebra $\mathbf X$, and $H$ is discretely embedded into $X$ then we denote it by $\mathbf H\leq_d\mathbf X$.
We denote by $\mathbb R$ and $\mathbb Z$ the additive group of the reals and the integers, respectively. 
Since linearly ordered abelian groups are exactly {\em cancellative} odd FL$_e$-chains \cite{Jenei_Hahn}
often we shall view $\mathbb R$ and $\mathbb Z$ as odd FL$_e$-chains.
\end{definition}
Crucial for our purposes will be the so-called {\em partial lexicographic product} construction. 
Denote the lexicographic product of two linearly ordered sets $X$ and $Y$ by $X\lex Y$.


\begin{definition}\label{FoKonstrukcio}{\rm \cite{Jenei_Hahn}}
\rm
Let ${\mathbf X}=(X, \wedge_X,\vee_X, \ast, \ite{\ast}, t_X, f_X)$ be an odd FL$_e$-algebra
and ${\mathbf Y}=( Y, \wedge_Y,\vee_Y, \star, \ite{\star}, t_Y, f_Y )$
be an involutive FL$_e$-algebra, with residual complement $\kompM{\ast}$ and $\kompM{\star}$, respectively.

\medskip
\begin{enumerate}
\item[A.]
Add a new element $\top$ to $Y$ as a top element and annihilator (for $\star$), then add a new element $\bot$ to $Y\cup\{\top\}$ as a bottom element and annihilator.
Extend $\kompM{\star}$ by $\negaM{\star}{\bot}=\top$ and $\negaM{\star}{\top}=\bot$.
Let 
$\mathbf W\leq\mathbf V\leq\mathbf X_\mathbf{gr}$. 
Let 
$$
\PLPIII{X}{V}{W}{Y}= (W\times (Y\cup\{\top,\bot\}))\cup ((V\setminus W)\times\{\top,\bot\})\cup \left((X\setminus V)\times \{\bot\}\right),
$$
and let $\PLPIII{\mathbf X}{\mathbf V}{\mathbf W}{\mathbf Y}$, the {\em type III partial lexicographic product} of $\mathbf X,\mathbf V,\mathbf W$ and $\mathbf Y$ be given by
$$\PLPIII{\mathbf X}{\mathbf V}{\mathbf W}{\mathbf Y}=\left(\PLPIII{X}{V}{W}{Y}, \leq, \te, \ite{\te}, (t_X,t_Y),(f_X,f_Y)\right),$$
where $\leq$ is the restriction of the lexicographical order of $\leq_X$ and $\leq_{Y\cup\{\top,\bot\}}$ to $\PLPIII{X}{V}{W}{Y}$, 
$\te$ is defined coordinatewise, and the operation $\ite{\te}$ is given by
$
\res{\te}{(x_1,y_1)}{(x_2,y_2)}=\nega{\left(\g{(x_1,y_1)}{\nega{(x_2,y_2)}}\right)} ,
$
where
$$
\nega{(x,y)}=\left\{
\begin{array}{ll}
(\negaM{\ast}{x},\bot) 		& \mbox{if $x\not\in V$}\\
(\negaM{\ast}{x},\negaM{\star}{y}) 	& \mbox{if $x\in V$}\\
\end{array}
\right. .
$$
In the particular case when $\mathbf W=\mathbf V$, we use the simpler notation 
$\PLPI{\mathbf X}{\mathbf V}{\mathbf Y}$ for $\PLPIII{\mathbf X}{\mathbf V}{\mathbf W}{\mathbf Y}$
and call it
the {\em type I partial lexicographic product} 
of $\mathbf X,\mathbf V$, and $\mathbf Y$.
\item[B.]\label{B}
Assume that $X_{gr}$ 
is discretely embedded into $X$.
Add a new element $\top$ to $Y$ as a top element and annihilator.
Let 
$\mathbf V\leq\mathbf X_\mathbf{gr}$. 
Let
$$
\PLPIV{X}{V}{Y}=(X\times \{\top\})\cup (V\times Y) 
$$
and let 
$\PLPIV{\mathbf X}{\mathbf V}{\mathbf Y}$, 
the {\em type IV partial lexicographic product} of $\mathbf X$, $\mathbf V$ and $\mathbf Y$ be given by
$$\PLPIV{\mathbf X}{\mathbf V}{\mathbf Y}=\left(\PLPIV{X}{V}{Y}, \leq, \te, \ite{\te}, (t_X,t_Y),(f_X,f_Y)\right),$$
where $\leq$ is the restriction of the lexicographical order of $\leq_X$ and $\leq_{ Y\cup\{\top\}}$ to 
$\PLPIV{X}{V}{Y}$,
$\te$ is defined coordinatewise, and the operation $\ite{\te}$ is given by
$
\res{\te}{(x_1,y_1)}{(x_2,y_2)}=\nega{\left(\g{(x_1,y_1)}{\nega{(x_2,y_2)}}\right)} ,
$
where $\komp$ is defined coordinatewise\footnote{Note that intuitively it would make up for a coordinatewise definition, too, in the second line of (\ref{FuraNegaEXT}) to define it as $(\negaM{\ast}{x},\bot)$. 
But $\bot$ is not amongst the set of possible second coordinates. However, since $X_{gr}$ is discretely embedded into $X$, if $(\negaM{\ast}{x},\bot)$ would be an element of the algebra then it would be equal to $((\negaM{\ast}{x})_\downarrow,\top)$. 
}
by
\begin{equation}\label{FuraNegaEXT}
\nega{(x,y)}=
\left\{
\begin{array}{ll}
(\negaM{\ast}{x},\top) 			& \mbox{if $x\not\in X_{gr}$ and $y=\top$}\\
((\negaM{\ast}{x})_\downarrow,\top) 	& \mbox{if $x\in X_{gr}$ and $y=\top$}\\
(\negaM{\ast}{x},\negaM{\star}{y}) 	& \mbox{if $x\in V$ and $y\in Y$}\\
\end{array}
\right. .
\end{equation}
In the particular case when $\mathbf V=\mathbf X_{\mathbf{gr}}$, we use the simpler notation 
$\PLPII{\mathbf X}{\mathbf Y}$ for $\PLPIV{\mathbf X}{\mathbf V}{\mathbf Y}$
and call it
the {\em type II partial lexicographic product} 
of $\mathbf X$ and $\mathbf Y$.
\end{enumerate}
\end{definition}

\begin{theorem}\label{SubLexiTheo}{\rm \cite{Jenei_Hahn}}
Adapt the notation of Definition~\ref{FoKonstrukcio}.
$\PLPIII{\mathbf X}{\mathbf V}{\mathbf W}{\mathbf Y}$
and
$\PLPIV{\mathbf X}{\mathbf V}{\mathbf Y}$
are
involutive FL$_e$-algebras with the same rank\footnote{The rank of an involutive FL$_e$-algebra is positive if $t>f$, negative if $t<f$, and $0$ if $t=f$.} as that of $\mathbf Y$.
In particular, if $\mathbf Y$ is odd then so are 
$\PLPIII{\mathbf X}{\mathbf V}{\mathbf W}{\mathbf Y}$
and
$\PLPIV{\mathbf X}{\mathbf V}{\mathbf Y}$.
In addition, 
$\PLPIII{\mathbf X}{\mathbf V}{\mathbf W}{\mathbf Y}\leq\PLPI{\mathbf X}{\mathbf V}{\mathbf Y}$ and
$\PLPIV{\mathbf X}{\mathbf V}{\mathbf Y}\leq\PLPII{\mathbf X}{\mathbf Y}$.
\end{theorem}

The main theorem on which we shall rely on in the present paper asserts that up to isomorphism, any odd FL$_e$-chain which has only finitely many positive idempotent elements can be built by iterating finitely many times the type III and type IV partial lexicographic product constructions 
using only linearly ordered abelian groups, as building blocks.

\medskip
A half-line proof shows (see \cite[Section 2]{Jenei_Hahn}) that in any odd FL$_e$-chain (in particular, in any group-like uninorm) the residual complement of any negative idempotent element is a positive idempotent element. Therefore, for a group-like uninorm, having finitely many idempotent elements is equivalent to having finitely many positive idempotent elements.

\begin{theorem}\label{Hahn_type}
If $\mathbf X$ is an odd FL$_e$-chain, which has only $n\in\mathbf N$, $n\geq 1$ idempotents in its positive cone then there exist linearly ordered abelian groups $\mathbf G_i$ $(i\in\{1,2,\ldots,n\})$, 
$\mathbf W_1\leq\mathbf V_1\leq \mathbf G_1$, $\mathbf W_i\leq\mathbf V_i\leq\mathbf W_{i-1}\lex\mathbf G_i$ $(i\in\{2,\ldots,n-1\})$,
and a binary sequence $\iota\in \{III,IV\}^{\{2,\ldots,n\}}$
such that 
$\mathbf X\simeq\mathbf X_n$, where
$\mathbf X_1:=\mathbf G_1$ and for $i\in\{2,\ldots,n\}$,
\begin{equation}\label{EzABeszed}
\mathbf X_i:=
\left\{
\begin{array}{ll}
\PLPIII{\mathbf X_{i-1}}{\mathbf V_{i-1}}{\mathbf W_{i-1}}{\mathbf G_i}	& \mbox{ if $\iota_i=III$}\\
\PLPIV{\mathbf X_{i-1}}{\mathbf V_{i-1}}{\mathbf G_i} 	& \mbox{ if $\iota_i=IV$}\\
\end{array}
\right. .
\end{equation}
\end{theorem}

\begin{lemma}\label{AB-C=A-BC}{\rm \cite{Jenei_IULfp}}
For any odd FL$_e$-algebras $\mathbf A$, $\mathbf B$, $\mathbf C$, it holds true that
$$
\PLPII{(\PLPII{\mathbf A}{\mathbf B})}{\mathbf C}\simeq\PLPII{\mathbf A}{(\PLPII{\mathbf B}{\mathbf C})},
$$ 
that is, if the algebra on one side is well-defined then the algebra on the other side is well-defined, too, and the two algebras are isomorphic.
\end{lemma}

\section{Structural description}

\begin{proposition}\label{TUTTItype12}
The following statements hold true.
\begin{enumerate}
\item 
Any type III extension which is not of type I has a gap\footnote{Gap refers to two consecutive elements, that is $x\leq y$ such that if $x\leq z\leq y$ then $z=x$ or $z=y$.} outside its group part.
\item 
Any type IV extension which is not of type II has a gap outside its group part.
\item 
If an odd FL$_e$-algebra has a gap outside its group part then
any type I extension of it has a gap outside its group part, too.
\item 
If an odd FL$_e$-algebra has a gap outside its group part then
any type II extension of it has a gap outside its group part, too.
\end{enumerate}
\end{proposition}
\begin{proof}
The statements are direct consequences of the definition of partial lexicographic products:
\begin{enumerate}
\item 
Consider $\PLPIII{\mathbf A}{\mathbf H}{\mathbf K}{\mathbf B}$, where $H\setminus K\neq\emptyset$. For any $a\in H\setminus K$ it holds true 
that $(a,\bot)<(a,\top)$ is a gap in $\PLPIII{\mathbf A}{\mathbf H}{\mathbf K}{\mathbf B}$, and neither $(a,\bot)$ nor $(a,\top)$ is invertible.
\item 
Consider $\PLPIV{\mathbf A}{\mathbf H}{\mathbf B}$, where $A_{gr}\setminus H\neq\emptyset$. For any $a\in A_{gr}\setminus H$ it holds true 
that $(a,\top)<(a_\uparrow,\top)$ is a gap in $\PLPIII{\mathbf A}{\mathbf H}{\mathbf K}{\mathbf B}$, and neither $(a,\top)$ nor $(a_\uparrow,\top)$ is invertible.

\item 
Consider $\PLPI{\mathbf A}{\mathbf H}{\mathbf B}$ with a gap $r<s$ in $A\setminus A_{gr}$.
Then 
$(r,\bot)<(s,\bot)$ is a gap in $\PLPI{\mathbf A}{\mathbf H}{\mathbf B}$, and neither $(r,\bot)$ nor $(s,\bot)$ is invertible.
\item 
Consider $\PLPII{\mathbf A}{\mathbf B}$ with a gap $r<s$ in $A\setminus A_{gr}$.
Then 
$(r,\top)<(s,\top)$ is a gap in $\PLPI{\mathbf A}{\mathbf H}{\mathbf B}$, and neither $(r,\top)$ nor $(s,\top)$ is invertible.
\end{enumerate}
\qed\end{proof}

\begin{lemma}\label{Vegyes}
For any odd FL$_e$-algebras $\mathbf A$, $\mathbf H$, $\mathbf L$, $\mathbf B$, such that 
$\mathbf H\leq\mathbf A_\mathbf{gr}$, 
it holds true that
$$\PLPII{(\PLPI{\mathbf A}{\mathbf H}{\mathbf L})}{\mathbf B} \simeq \PLPI{\mathbf A}{\mathbf H}{(\PLPII{\mathbf L}{\mathbf B})}$$ 
that is, if the algebra on one side is well-defined then the algebra on the other side is well-defined, too, and the two algebras are isomorphic.
\end{lemma}
\begin{proof}
By definition, the left-hand side is well-defined if and only if 

(i)
$H\times L_{gr}$ is discretely embedded into $
(H\times L)\cup(A\times\{\bot_L\})$.
\\
The right-hand side is well-defined if and only if 

(ii)
$L_{gr}$ is discretely embedded into $L$.
\\
Clearly, (ii) implies (i).
Now, assume (i).
Then $L_{gr}$ cannot be finite. Indeed, if $L_{gr}$ were finite then by taking its largest element $l\in L_{gr}$, 
the element $(1_H,l)_\uparrow$ must be greater than $(1_H,l)$ since $H\times L_{gr}$ is discretely embedded into $(H\times L)\cup(A\times\{\bot_L\})$.
Therefore, $(1_H,l)_\uparrow$ is either equal to $(1_H,l_\uparrow)$ which is not in $H\times L_{gr}$ since $l$ was chosen the greatest element in $L_{gr}$,
or equal to $(1_H,\top)$ which is not in $H\times L_{gr}$ either.
Thus, $L_{gr}$ is infinite, and hence for any $(h,l)\in H\times L_{gr}$ it holds true that $(h,l)_\uparrow=(h,l_\uparrow)\in H\times L_{gr}$ and $(h,l)_\downarrow=(h,l_\downarrow)\in H\times L_{gr}$, that is, (ii) holds.

Denote $\mathbf C=\PLPI{\mathbf A}{\mathbf H}{\mathbf L}$ for short.
By definition, 
$C=(H\times (L\cup\{\top_L\})\cup(A\times\{\bot_L\})$,
therefore the universe of the left-hand side is
$$
(C_{gr}\times B)\cup(C\times \{\top_{B}\})
=
$$
$$
[H\times L_{gr}\times B]
\cup
[H\times L\times \{\top_{B}\}]
\cup
[H\times \{\top_L\}\times \{\top_{B}\}]
\cup
[A\times\{\bot_L\}\times \{\top_{B}\}]
$$
On the other hand, the universe of $\PLPII{\mathbf L}{\mathbf B}$ is 
$
(L_{gr}\times B)\cup(L\times\{\top_B\})
$.
Let $\top_L$, $\bot_L$, and $\top_B$ be the new top, bottom, and top element added to $L$, $L$, and $B$, respectively as in item A of Definition~\ref{FoKonstrukcio}. 
Then it is easily verified that $(\top_L,\top_B)$ and $(\bot_L,\top_B)$ satisfy the requirements of Definition~\ref{FoKonstrukcio} to be the new top and bottom elements of $\PLPII{\mathbf L}{\mathbf B}$. 
Hence the universe of the right-hand side is
$$
[H\times L_{gr}\times B]
\cup
[H\times L\times \{\top_{B}\}]
\cup
[H\times \{\top_L\}\times \{\top_{B}\}]
\cup
[A\times\{\bot_L\}\times \{\top_{B}\}]
$$
so the underlying universes of the two sides coincide.
Clearly, the unit elements are the same. 
Since the monoidal operation of a partial lexicographic product is defined coordinatewise, the respective monoidal operations coincide, too.
Since both algebras are residuated and the monoidal operation uniquely determines its residual operation, it follows that the residual operations coincide, too, hence so do the residual complements.
\qed\end{proof}
If the underlying universes of two odd FL$_e$-chains $\mathbf X$, $\mathbf Y$ are order isomorphic then we will denote it by  $\mathbf X\simeq_o \mathbf Y$

\begin{lemma}\label{topbotJOLESZ}
Let $\mathbf A$ and $\mathbf D$ be odd FL$_e$-chains, $
\mathbf H\leq\mathbf A_\mathbf{gr}$.
The following statements are equivalent.
\begin{enumerate}
\item
$\PLPI{\mathbf A}{\mathbf H}{\mathbf D}\simeq_o\mathbb{R}$,
\item
$\mathbf A\simeq_o\mathbb{R}$, $\mathbf D\simeq_o\mathbb{R}$, and $\mathbf H$ is countable.
\end{enumerate}
\end{lemma}
\begin{proof}
Sufficiency has been proved in \cite[Proposition 8]{Jenei_IULfp}.
Denote $\mathbf C=\PLPI{\mathbf A}{\mathbf H}{\mathbf D}$ for short.
To prove the necessity, assume $\mathbf C\simeq_o\mathbb{R}$.

(i)
If $D$ had a least element $l$ then for some $h\in H$, $(h,\bot)<(h,l)$ would make a gap in $C$, a contradiction. An analogous argument shows that $D$ cannot have a greatest element either.
Next, if $A$ had a least or a greatest element ($l$ or $g$) then $(l,\bot)$ or $(g,\bot)$ would be the least or the greatest element of $C$, a contradiction.

(ii)
If $D$ is not densely ordered then there exists a gap $a<b$ in $D$. Then for $h\in H$, $(h,a)<(h,b)$ is a gap in $C$, a contradiction.
If $A$ is not densely ordered then there exists a gap $a<b$ in $A$. Then
$(a,\top)<(b,\bot)$ is a gap in $C$ when $a\in H$, and
$(a,\bot)<(b,\bot)$ is a gap in $C$ when $a\in A\setminus H$, contradiction.

(iii)
Let $Q_2$ be a countable and dense subset of $C$.

We prove that $Q:=\{ v_1\in A \ | \ (v_1,v_3)\in Q_2 \}$ is a countable and dense subset of $A$.
Indeed, $Q$ is clearly nonempty and countable, too, since so is $Q_2$.
To show that $Q$ is a dense subset of $A$, let $a\in A\setminus Q$ arbitrary.
Take an arbitrary open interval $]b,c[$ containing $a$.
Since $Q_2$ is a dense subset of $C$, we can choose an element $(r,s)\in Q_2$ such that $(b,\bot)<(r,s)<(c,\bot)$ if $b\in A\setminus H$, or we can choose $(r,s)\in Q_2$ such that $(b,\top)<(r,s)<(c,\bot)$ if $b\in H$.
In both cases, $r\in Q$ and $b<r<c$ holds, so we are done.

Next, we prove that $H$ is countable.
It suffices to prove $H\subseteq Q$ since $Q$ is a countable.
Assume there exists $h\in H\setminus Q$.
Then $\{h\}\times(D\cup\{\top,\bot\})\subset C\setminus Q_2$ would follow, showing that there is no element in $Q_2$ between $(h,\bot)$ and $(h,\top)$. However there should be, since $D$ is nonempty.
It is a contradiction to $Q_2$ being a dense subset of $C$.


Finally, we prove that $Q_3:=\{ v_3\in D \ | \ (v_1,v_3)\in Q_2 \}$ is a countable and dense subset of $D$.
Indeed, $Q_3$ is clearly nonempty and countable, too, since so is $Q_2$.
To show that $Q_3$ is a dense subset of $D$, let $a_3\in D\setminus Q_3$ arbitrary.
Take an arbitrary open interval $]b_3,c_3[$ containing $a_3$, and let $h\in H$.
Since $Q_2$ is a dense subset of $C$
and $(h,b_3),(h,c_3)\in C$, we can choose an element $(h,s_3)\in Q_2$ such that 
$(h,b_3)<(h,s_3)<(h,c_3)$. Thus $s_3\in Q_3$ follows and $b_3<s_3<c_3$ holds, so we are done.

(iv)
To prove that $A$ is Dedekind complete we proceed as follows.
Take any nonempty subset of $V\subseteq A$ which has an upper bound $b\in A$.
Then $V_2=\{(v,\bot) \ | \ v\in V\}$ is a nonempty subset of $C$ which has an upper bound $(b,\bot)\in C$. Since $C$ is Dedekind complete, there exists the supremum $(m,m_3)$ of $V_2$ in $C$.
Because the second coordinate of any element of $V_2$ is $\bot$, it follows that when $(m,m_3)$ is an upper bound of $V_2$ then also $(m,\bot)$ is an upper bound of it. Therefore, $(m,\bot)\in C$ is the supremum of $V_2$, and hence $m$ is the supremum of $V$.

To prove that $D$ is Dedekind complete we proceed as follows.
Take any nonempty subset of $V_3\subseteq D$ which has an upper bound $b_3\in D$.
Choose an element $h$ from $H$. 
Then $V_2=\{(h,v_3) \ | \ v_3\in V_3\}$ is a nonempty subset of $C$ which has an upper bound $(h,b_3)\in C$. Since $C$ is Dedekind complete, there exists the supremum $(m,m_3)$ of $V_2$ in $C$.
Clearly, $(m,m_3)\leq(h,b_3)$ holds, therefore, $m=h$ follows, and since $(h,m_3)$ is the supremum of $V_2$, it follows that $m_3$ is the supremum of $V_3$.
\qed\end{proof}

\color{black}
\begin{lemma}\label{sdhajkdhasgGGG}
Let $\mathbf A$, $\mathbf B$ and $\mathbf L$ be odd FL$_e$-chains, $
\mathbf H\leq\mathbf A_\mathbf{gr}$. 
The following statements are equivalent.
\begin{enumerate}
\item
$
\PLPII{(\PLPI{\mathbf A}{\mathbf H}{\mathbf L})}{\mathbf B}
$
is well-defined and $\simeq_o\mathbb{R}$
\item
\begin{itemize}
\item
$\mathbf A\simeq_o\mathbb R$ and $\mathbf B\simeq_o\mathbb R$,
\item
$H$ and $L_{gr}$ are countable, 
\item
$L$ is Dedekind complete, has a countable dense subset, and has neither least nor greatest element, 
\item
$L_{gr}$ is discretely embedded into $L$,
and 
\item
there exists no gap in $L$ formed by two elements of $L\setminus L_{gr}$.
\end{itemize}
\end{enumerate}
\end{lemma}
\begin{proof}
Sufficiency has been proved in \cite[Proposition 6]{Jenei_IULfp}.
Denote $\mathbf D=\PLPII{(\PLPI{\mathbf A}{\mathbf H}{\mathbf L})}{\mathbf B}$ for short.
To prove the necessity, assume that $\mathbf D$ is well-defined and is $D\simeq_o\mathbb{R}$.
\begin{itemize}
\item
First we prove that $\mathbf A\simeq_o\mathbb R$ and $\mathbf B\simeq_o\mathbb R$.

- Since partial lexicographic products clearly inherit the boundedness of their first component, and since $D$ has neither least nor greatest element, it follows that $C$, and in turn, $A$ has neither least nor greatest element.
If $B$ had a greatest element $g$ then, since $t_{C}\in C_{gr}$, it would yield $(t_{C},g)<(t_{C},\top_{B})$ be a gap in $D$, a contradiction. 
Since $B$ is involutive, it cannot have a least element either, because then it would have a greatest one, too.

- We prove that $A$ and $B$ are densely ordered.
Assume $A$ isn't. Then there exists a gap $a<b$ in $A$.
If $a\in H$ then $(a,\top_L)<(b,\bot_L)$ is a gap in $C$, whereas if $a\in A\setminus H$ then $(a,\bot_L)<(b,\bot_L)$ is a gap in $C$.
In all cases, there is a gap $c<d$ in $C$ such that $c,d\in C\setminus C_{gr}$, thus yielding a gap $(c,\top_{B})<(d,\top_{B})$ in $D$, a contradiction.
If $B$ were not densely ordered witnessed by a gap $a_4<b_4$ then for  $t_{C}\in C_{gr}$, $(t_{C},a_4)<(t_{C},b_4)$ would be a gap in $D$, a contradiction.

- Next we prove that $A$ and $B$ are Dedekind complete.
First assume $B$ isn't. Then there exists a nonempty subset $X_4$ of $B$ bounded above by $b_4\in B$ such that $X$ does not have a supremum in $B$.
Then for any $a_2\in C_{gr}$, the set $\{(a_2,x) \ | \ x\in X_4\}\subseteq D$ is nonempty, it is bounded from above by $(a_2,b_4)$, and it does not have a supremum in $D$, a contradiction.
Second, we assume that $A$ is not Dedekind complete, that is, there exists a nonempty subset $X$ of $A$ bounded above by $b\in A$ such that $X$ does not have a supremum in $A$.
Let $X_3=X\times\{\bot_L\}\times\{ \top_{B} \}$. Then $\emptyset\not=X_3\subseteq D$, $X_3$ is bounded from above by $(b,\bot_L, \top_{B})$.
Let $(c,d,e)\in D$ be an upper bound of $X_3$.
Clearly, $(c,\bot_L,\top_{B})$ is an upper bound of $X_3$, too, hence $c\in A$ is an upper bound of $X$.
Therefore, there exists $A\ni s<c$ such that also $s$ is an upper bound of $X$. Thus $(s,\bot_L,\top_{B})<(c,\bot_L,\top_{B})$ is an upper bound of $X_3$, too, showing that $D$ is not Dedekind complete, a contradiction.

- We prove that both $A$ and $B$ have a countable and dense subset.
Let $D_3$ be a countable and dense subset of $D$.
Let $D=\{ a \ | \ (a,l,a_4)\in D_3 \}$. We claim that $D$ is a countable and dense subset of $A$.
Indeed, $D$ is nonempty and countable, since so is $D_3$. Assume that there exists $d\in A\setminus D$ such that $d$ is not an accumulation point of $D$.
Since $A$ has neither least nor greatest element, it follows that there exists $b,c\in A$ such that $b<d<c$ and there is no element of $D$ in between $b$ and $c$.
Then it follows that $(d,\bot_L,\top_{B})\in D\setminus D_3$ is not an accumulation point of $D_3$, as shown by the neighborhood $D\ni(b,\bot_L,\top_{B})<(d,\bot_L,\top_{B})<(c,\bot_L,\top_{B})\in D$, a contradiction to $D_3$ being a dense subset of $D$.
Next, we claim that $D_4=\{  a_4\in B \ | \ (t_{C},a_4)\in D_3 \}$ is a countable and dense subset of $B$.
Indeed, $D_4$ is countable, since so is $D_3$.
$D_4$ is nonempty, since if for any $a_4\in B$, $(t_{C},a_4)\notin D_3$ then 
using that $B$ has neither least nor greatest element and thus $B$ is infinite, it follows that there exists $s,v,w\in B$ such that $(t_{C},s)<(t_{C},v)<(t_{C},w)$, showing that $(t_{C},v)\in D\setminus D_3$ is not and accumulation point of $D_3$, a contradiction.
Assume that there exists $d_4\in B\setminus D_4$ such that $d_4$ is not an accumulation point of $D_4$.
Since $B$ has neither least nor greatest element, it follows that there exists $b_4,c_4\in B$ such that $b_4<d_4<c_4$ and there is no element of $D_4$ in between $b_4$ and $c_4$.
Then it follows that $(t_{C},d_4)\in D\setminus D_3$ is not an accumulation point of $D_3$, witnessed by the neighborhood $D\ni(t_{C},b_4)<(t_{C},d_4)<(t_{C},c_4)\in D$, a contradiction to $D_3$ being a dense subset of $D$.
\item
We prove that $H$ and $L_{gr}$ are countable.
Assume than any of them isn't. Then, since $C_{gr}=H\times L_{gr}$, it follows that $C_{gr}$ is uncountable, too.
In the preceding item, we proved that $D_4$ is nonempty.
In complete analogy, we can prove that for any $c_2\in C_{gr}$, $D_{c_2}=\{  a_4\in B \ | \ (c_2,a_4)\in D_3 \}$ is nonempty either.
But it means that for any $c_2\in C_{gr}$, there is an element $(c_2,y_{c_2})$ in $D_3$.
Since $c_2\mapsto(c_2,y_{c_2})$ is injective, it follows that $D_3$ is uncountable, a contradiction.
\item
Finally we prove the statements about $L$.

- Since $\PLPII{(\PLPI{\mathbf A}{\mathbf H}{\mathbf L})}{\mathbf B}$ is well-defined, $\left(\PLPI{A}{H}{L}\right)_{gr}=H\times L_{gr}$ is discretely embedded into $(H\times (L\cup\top_L))\cup(A\times\{\bot_L\})$.
Let $l\in L_{gr}$ be arbitrary. Then $H\times L_{gr}\ni(t_H,l)_\uparrow$ cannot be $(t_H,\top_L)$ since it is not in $H\times L_{gr}$. Therefore, $(t_H,l)_\uparrow$ is equal to $(t_H,l_\uparrow)$ and it is in $H\times L_{gr}$. Thus, $l_\uparrow\in L_{gr}$. Summing up, $L_{gr}$ is discretely embedded into $L$.

- If $L$ had a greatest element $g$ then $(t_H,g,\top_{B})<(t_H,\top_L,\top_{B})$ were a gap in $D$, a contradiction.
Since $L$ is involutive, it cannot have a least element either, because then it would have a greatest one, too.

- Next we prove that $L$ is Dedekind complete.
Let an arbitrary $\emptyset\not=L_1\subset L$ be bounded above by $l\in L$.
Then $\{(t_H,l_1,\top_{B})\ | \ l_1\in L_1\}\subset D$ is nonempty, it is bounded from above by $(t_H,l,\top_{B})$, and since $D$ is Dedekind complete, there exists a supremum $(x,y,z)$ of it in $D$.
Clearly, $x=t_H$ and for any $l_1\in L_1$, $\l_1\leq y\leq l$ holds. The latest implies $y\in L$.
But then $y$ is the supremum of $L_1$. Indeed, if for any $l_1\in L_1$, $\l_1\leq z< y$ would hold then $(t_H,z,\top_{B})$ would also be an upper bound of $L_1$, a contradiction.

- If there were a gap $l_1<l_2$ in $L$ formed by two non-invertible elements then $D\supset H\times (L\setminus L_{gr})\times \{\top_{B}\}\ni (t_H,l_1,\top_{B})<(t_H,l_2,\top_{B})\in H\times (L\setminus L_{gr})\times \{\top_{B}\}\subset D$ would be a gap in $D$, a contradiction.

- Let $D_L=\{ l\in L \ | \ (a,l,a_4)\in D_3 \}$.
We prove that $D_L$ is a countable and dense subset of $L$. 
Indeed, $D_L$ is clearly countable since so is $D_3$, and $D_L\subseteq [H\times L_{gr}\times (B\cup \{\top_{B}\})]
\cup
[H\times (L\setminus L_{gr})\times \{\top_{B}\}]$ holds.
Assume that there is $l_1\in D_L\setminus L$ such that $l_1$ is not an accumulation point of $D_L$.
Since $L$ has neither least nor greatest element, there is $s,v\in L$ such that $s<l_1<v$ and there is no element of $D_L$ strictly in between $s$ and $v$.
If $l_1\in L_{gr}$ then choose $a,b,c\in B$ such that $a<b<c$\,;
then $(t_H,l_1,b)\in D$ is not an accumulation point of $D_3$ witnessed by its neighborhood 
$D\ni (t_H,l_1,a)<(t_H,l_1,c)\in D$, a contradiction.
Hence we can assume $l_1\in L\setminus L_{gr}$. 
If $v\in L\setminus L_{gr}$ then there exists $w\in L$ such that $l_1<w<v$ since there exists no gap in $L$ formed by two elements of $L\setminus L_{gr}$, 
whereas 
if 
$v\in L_{gr}$ then $w:=v_\downarrow<v$ holds since $L_{gr}$ is discretely embedded into $L$, and $L\setminus L_{gr}\ni l_1\not=v_\downarrow\in L_{gr}$.
In both cases $s<l_1<w<v$ follows.
Therefore, $(t_H,l_1,\top_{B})$ is not an accumulation point of $D_3$ witnessed by its neighborhood 
$D\ni (t_H,s,\top_{B})<(t_H,w,\top_{B})\in D$, a contradiction.
\end{itemize}
\qed\end{proof}


\begin{lemma}\label{EZaz_Z}
Any linearly ordered abelian group $\mathbf G$ which is 
Dedekind complete and satisfies $x_\downarrow<x<x_\uparrow$ 
is isomorphic 
(qua an FL$_e$-algebra)
to $\mathbb Z$.
\end{lemma}
\begin{proof}
First we prove that $G$ is archimedean.
If not then there exists $x,y\in G^+$ such that for any $n\in\mathbb N$, $a^{(n)}<b$. 
Then $X=\{  a^{(n)} \ | \ n\in\mathbb N \}$ is bounded from above by $b$, and hence, using that $G$ is Dedekind complete, it has a least upper bound $m$.

It holds true that $a^{(n)}$ is strictly increasing since $a\in G^+$:
$a^{n+1}=\g{a^n}{a}\geq\g{a^n}{1}=a^n$ holds since $a$ is positive and $\te$ is increasing in both arguments.
If $a^{n+1}$ were equal to $a^n$ for some $n\in\mathbb N$ then a cancellation by $a^n$ would imply $a=1$, a contradiction to $a>1$. 

We obtain $m\notin X$ since $a_n=m$ for some $n\in\mathbb N$ would then imply $a_{n+1}>m$, a contradiction to $m$ being an upper bound of $X$.
By the hypothesis $m_\downarrow<m$ holds, and we state that $m_\downarrow$ is an upper bound of $X$, too, otherwise there were an element $a_n\in X$ such that $a_n>m_\downarrow$, that is, $a_n\geq m$, that is, $a_n=m$ since $m$ is an upper bound of $X$, a contradiction.
This contradicts to $m$ being the least upper bound of $X$.
\\
By H\" older theorem, archimedean linearly ordered abelian groups are embeddable into $\mathbb R$. Invoking the last assumption of the lemma, it follows that $G$ is isomorphic to a discretely ordered subgroup of $\mathbb R$, hence it is isomorphic to $\mathbb Z$.
The isomorphism naturally extends to an isomorphism between the respective FL$_e$-algebras.
\qed\end{proof}

\begin{lemma}\label{AlappalIZO}
Let $\mathbf A$ be a linearly ordered abelian group, $\mathbf B$ be an odd FL$_e$-chain.
Then $\PLPII{\mathbf A}{\mathbf B}\simeq_o\mathbb R$ if and only if
(qua FL$_e$-algebras)
$\mathbf A\simeq\mathbb Z$ and $\mathbf B\simeq_o\mathbb R$.
\end{lemma}
\begin{proof}
Denote $\mathbbm{1}$ the trivial one-element subalgebra of $\mathbb R$.
By Lemma~\ref{topbotJOLESZ}, $\PLPI{\mathbb R}{\mathbbm 1}{(\PLPII{\mathbf A}{\mathbf B})}\simeq_o\mathbb R$.
Therefore, by Lemma~\ref{Vegyes}, 
$\PLPII{(\PLPI{\mathbb R}{\mathbbm 1}{\mathbf A})}{\mathbf B}\simeq_o\mathbb R$.
By Lemma~\ref{sdhajkdhasgGGG}, $\mathbf B\simeq_o\mathbb R$.
Also by Lemma~\ref{sdhajkdhasgGGG} and by using that $\mathbf A=\mathbf A_\mathbf{gr}$, 
it follows that 
$\mathbf A$ is Dedekind complete and that 
$A$ is discretely embedded into $A$. This latest condition is equivalent to saying that $x_\downarrow<x<x_\uparrow$ \ for $x\in A$.
Thus, by Lemma~\ref{EZaz_Z}, $\mathbf A\simeq\mathbb Z$.
\qed\end{proof}

\begin{lemma}\label{RazR}
For any abelian group $\mathbf G$, $\mathbf G\simeq\mathbb R$ if and only if $\mathbf G\simeq_o\mathbb R$.
\end{lemma}
\begin{proof}
Assume $\mathbf G\simeq_o\mathbb R$. Then $\mathbf G$ is archimedean.
Indeed, assume it is not. Then there exist two  elements $a,b>1$ in $\mathbf G$ such that for any $n\in\mathbb N$, $a^n<b$. 
Since $\mathbf G\simeq_o\mathbb R$, $\mathbf G$ is Dedekind complete. Therefore, since $\{a^n\ | \ n\in\mathbb N\}$ is bounded from above by $b$, it has a supremum $m$. We shall prove that $m$ is an idempotent element, and it is different from the unit element $1$, contradicting to $\mathbf G$ being a group.
(i) $m$ is idempotent:
$\g{m}{m}=\g{(\sup_{n\in\mathbb N} a^n)}{(\sup_{n\in\mathbb N} a^n)}$. Since $\mathbf G$ is residuated, $\te$ distributes over arbitrary joins.
Hence, $\g{(\sup_{n\in\mathbb N} a^n)}{(\sup_{n\in\mathbb N} a^n)}=\sup_{n\in\mathbb N}(\g{a^n}{a^n})=\sup_{n\in\mathbb N}a^{2n}=m$.
(ii) $m>1$: Since $a^0=1$ it is sufficient to prove that $a^n$ is strictly increasing, which holds true; adapt here the related proof of Lemma~\ref{EZaz_Z}.
Therefore, $\mathbf G$ is archimedean. By the H\" older theorem $\mathbf G$ embeds into the additive group of reals.
Now, as it is well-known, a subgroup of the additive group of reals is either isomorphic to the the additive group of integers, or it is a dense subset of the reals. 
The $\mathbf G\simeq_o\mathbb R$ condition leaves only the latter case, and Dedekind completeness implies that the universe of $\mathbf G$ is equal to the set of real numbers.
\qed\end{proof}

\begin{definition}\label{basicPL}\rm{\bf (Basic group-like uninorms)}
Let $\mathbb U_0=\mathbb R$ and for $n\in\mathbb N$
let
$\mathbb U_{n+1}=
\PLPII{\mathbb Z}{\mathbb U_n}$. 
Proposition~\ref{AB-C=A-BC} yields that it can equivalently be written without brackets as
$$
\mathbb U_n=
\PLPII{\underbrace{\PLPII{\mathbb Z}{\PLPII{\ldots}{\mathbb Z}}}_n}{\mathbb R}
.$$
\vskip-0.3cm
\begin{figure}
\begin{center}
  \includegraphics[width=0.44\textwidth]{Symm_Prod.pdf} \ \  
  \includegraphics[width=0.43\textwidth]{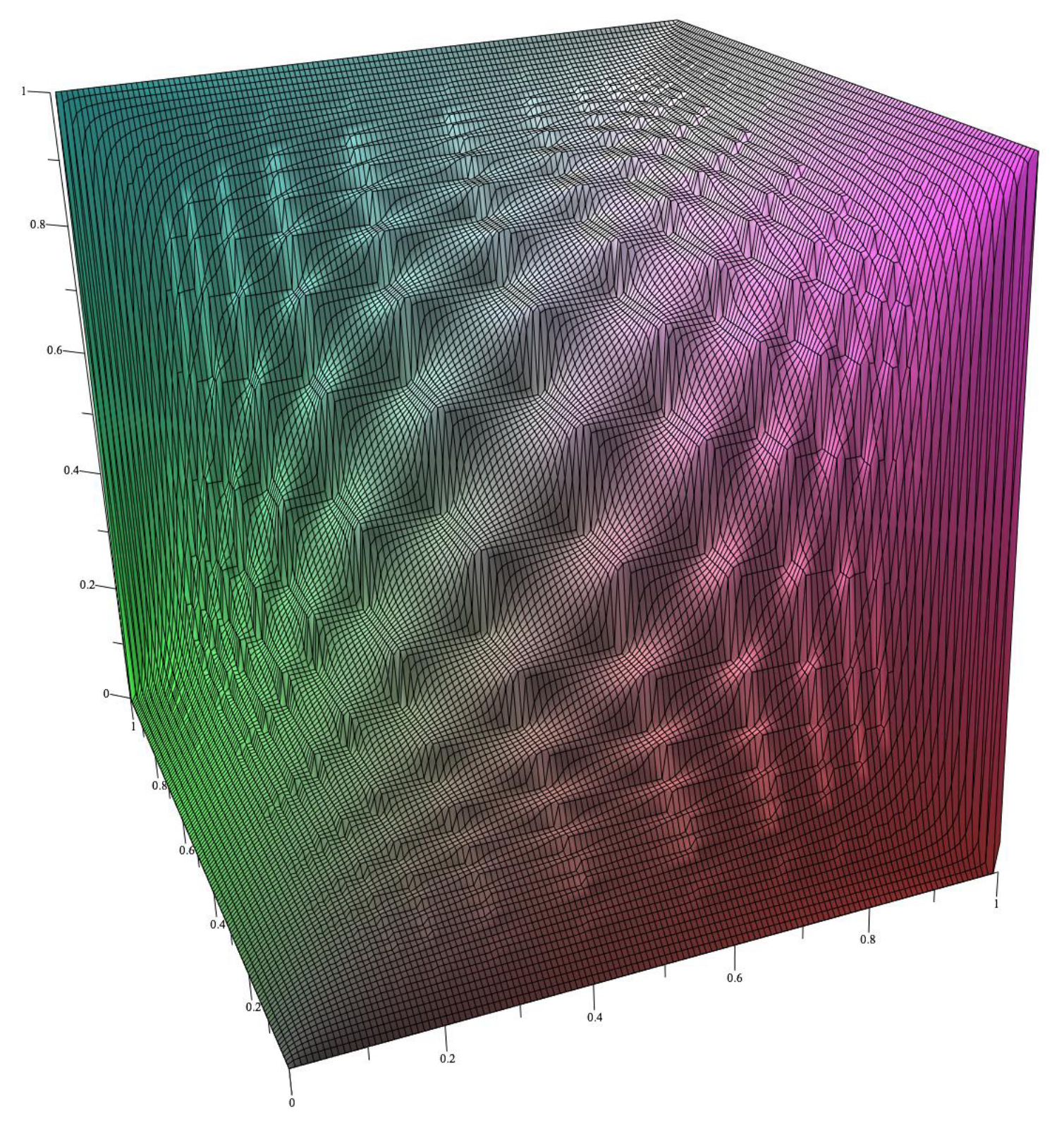} \ \ 
\caption{Visualization: Two basic group-like uninorms, 
$\mathbb U_0=\mathbb R$ and 
$\mathbb U_1=\PLPII{\mathbb Z}{\mathbb R}$
shrank into $]0,1[$.
One can describe $\mathbb U_1$ as infinitely many $\mathbb U_0$ components.
Immagine $\mathbb U_2$ in the same way: as infinitely many $\mathbb U_1$ components, etc.
}
\label{fig:2}
\end{center}
\end{figure}

\begin{figure}
\begin{center}
  \includegraphics[width=0.43\textwidth]{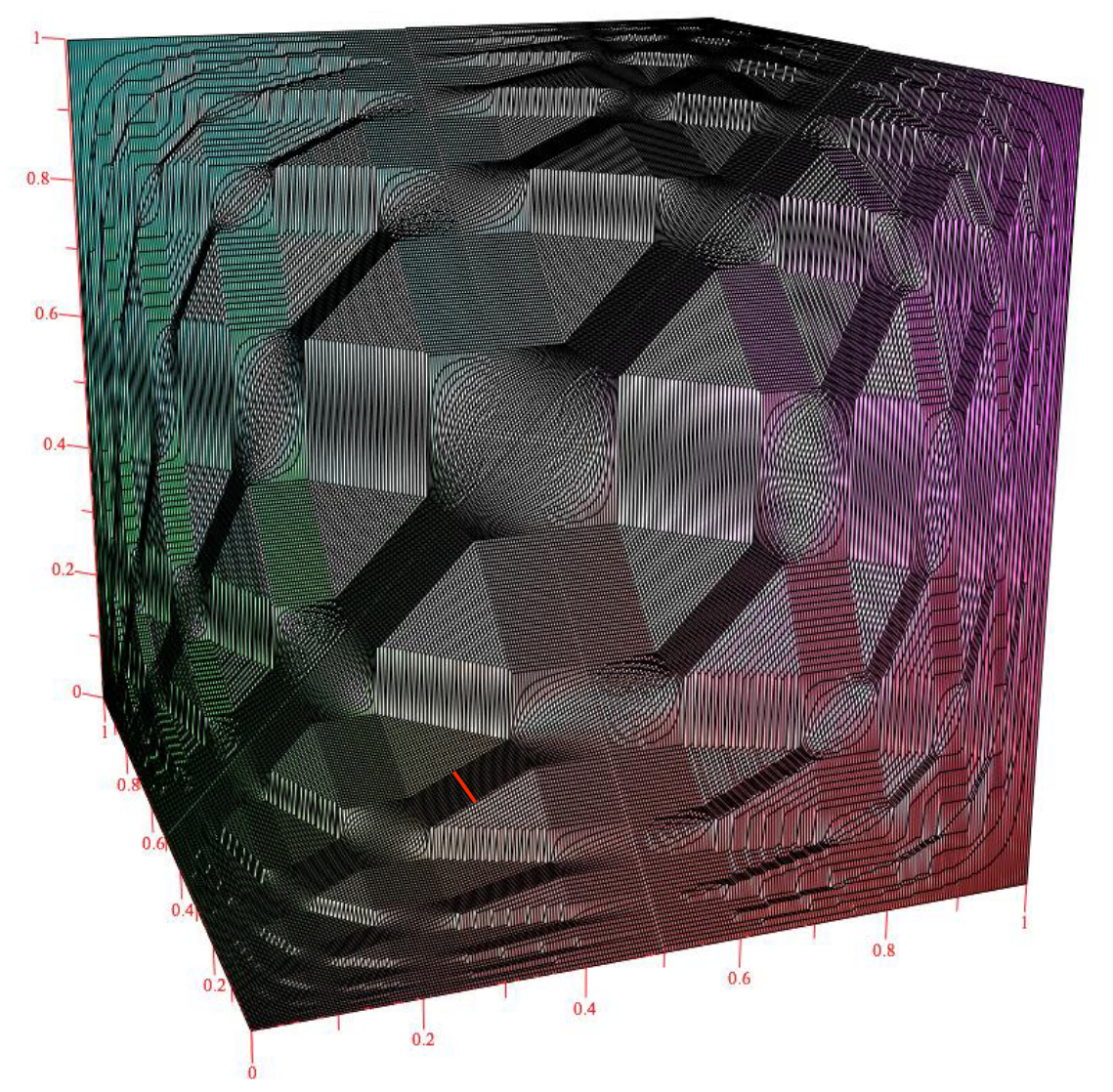} \ \ 
\caption{Visualization: An example for the first type extension,
$\PLPI{\mathbb R}{\mathbb Z}{\mathbb R}$ shrank into $]0,1[$
}
\label{fig:3}
\end{center}
\end{figure}
\end{definition}

\begin{lemma}\label{ZZZR}
Let $n\in\mathbb N$, $n\geq 1$. For $i=1,\ldots, n$, let $\mathbf G_i$ be linearly ordered abelian groups.
If 
$\mathbf U=\PLPII{\PLPII{\mathbf G_1}{\mathbf G_2}}{\PLPII{\ldots}{\mathbf G_n}}\simeq_o\mathbb R$ then 
$\mathbf G_i\simeq\mathbb Z$ holds for $1\leq i\leq n-1$, $\mathbf G_n\simeq\mathbb R$, and thus 
$\mathbf U\simeq \mathbb U_{n-1}$.
\end{lemma}
\begin{proof}
Induction on $n$.
If $n=1$ then $\mathbf G_1$ is a linearly ordered abelian group and $\mathbf G_1\simeq_o\mathbb R$. By Lemma~\ref{RazR}, $\mathbf G_1\simeq\mathbb R$.
The case $n=2$ is concluded by Lemmas~\ref{AlappalIZO} and \ref{RazR}.
Assume the statement holds for $k-1$. 
An application of Lemma~\ref{AlappalIZO} to 
$\PLPII{\mathbf G_1}{(\PLPII{\mathbf G_2}{\PLPII{\ldots}{\mathbf G_k}})}$
yields that 
(qua FL$_e$-algebras)
$\mathbf G_1\simeq\mathbb Z$ and $\PLPII{\mathbf G_2}{\PLPII{\ldots}{\mathbf G_k}}\simeq\mathbb R_o$.
By the induction hypothesis, for $2\leq i\leq k-1$, $\mathbf G_i\simeq\mathbb Z$ and $\mathbf G_n\simeq \mathbb R$, and we are done.
\end{proof}

We are ready to prove the main theorem:
Theorem~\ref{GrLikeUni} is a representation theorem for those group-like uninorms which has finitely many idempotent elements, by means of basic group-like uninorms and the type I partial-lexicographic product construction.
Alternatively, one may view Theorem~\ref{GrLikeUni} as a representation theorem for those group-like uninorms which has finitely many idempotent elements, by means of $\mathbb Z$ and $\mathbb R$ and the type I and type II partial-lexicographic product constructions.

\begin{theorem}\label{GrLikeUni}{{\bf (Representation by basic group-like uninorms)}} 
If $\mathbf U$ is a group-like uninorm, which has finitely many idempotent elements, out of which there are $m\in\mathbf N$, $m\geq 1$ idempotents in its negative cone then there exists a sequence $k\in  \mathbb N^{\{1,\ldots,m\}}$
such that
$\mathbf U\simeq\mathbf U_m$, where
for $i\in\{1,\ldots,m\}$,
$$
\mathbf U_i
=
\left\{
\begin{array}{ll}
\mathbb U_{k_1}  & \mbox{ if $i=1$}\\
\PLPI{\mathbf U_{i-1}}{\mathbf H_{i-1}}{\mathbb U_{k_i}}		 & \mbox{ if $2\leq i\leq m$}
\end{array}
\right.
,
$$
where for $2<i\leq m$, $\mathbf H_{i-1}$ is a countable subgroup of $(\mathbf U_{i-1})_\mathbf{gr}$.
\end{theorem}
\begin{proof}
Consider the group representation of $\mathbf U$ according to Theorem~\ref{Hahn_type}:
Since $\mathbf X_n\simeq\mathbf U\simeq_o\mathbb R$, $\mathbf X_n$ cannot have any gaps, and
by Proposition~\ref{TUTTItype12} it follows that 
all extensions in its group representation are either of type I or type II.
More formally, 
there exist linearly ordered abelian groups $\mathbf G_i$ $(i\in\{1,2,\ldots,n\})$, 
$\mathbf V_1\leq \mathbf G_1$, $\mathbf V_i\leq\mathbf V_{i-1}\lex\mathbf G_i$ $(i\in\{2,\ldots,n-1\})$,
and a binary sequence $\iota\in \{I,II\}^{\{2,\ldots,n\}}$
such that 
$\mathbf U\simeq\mathbf X_n$, where
$\mathbf X_1:=\mathbf G_1$ and for $i\in\{2,\ldots,n\}$,
\begin{equation}\label{EzABeszed}
\mathbf X_i
=
\left\{
\begin{array}{ll}
\PLPI{\mathbf X_{i-1}}{\mathbf V_{i-1}}{\mathbf G_i}	& \mbox{ if $\iota_i=I$}\\
\PLPII{\mathbf X_{i-1}}{\mathbf G_i} 	& \mbox{ if $\iota_i=II$}\\
\end{array}
\right. .
\end{equation}
Induction on $l$, the number of type I extensions in the group representation.

If $l=0$ then by Lemma~\ref{AB-C=A-BC} the brackets can be omitted, thus
$\mathbf U\simeq\PLPII{\PLPII{\mathbf G_1}{\mathbf G_2}}{\PLPII{\ldots}{\mathbf G_n}}$.
By denoting $k_1=n-1$, Lemma~\ref{ZZZR} confirms  $\mathbf U\simeq \mathbb U_{k_1}$.

Let $l\geq 1$ and assume that the statement holds for $l-1$, and that $\mathbf U$ has $l$ type I extensions in its group representation.
There are two cases:

If $\iota_n=I$ then 
$
\mathbf U\simeq
\PLPI{\mathbf X_{n-1}}{\mathbf H_{n-1}}{\mathbf G_n}
$
and by Lemma~\ref{topbotJOLESZ}, 
$\mathbf X_{n-1}\simeq_o\mathbb R$, $\mathbf G_n\simeq_o\mathbb R$, and $\mathbf H_{n-1}$ is countable.
By Lemma~\ref{RazR}, $\mathbf G_n\simeq\mathbb R=\mathbb U_0$. Applying the induction hypothesis to $\mathbf X_{n-1}$ concludes the proof.

If $\iota_n=II$ then 
let $j=\max\{i\in\{1,\ldots,n\}\ | \  \iota_i=I \}$.
Note that this set in nonempty, since $l\geq1$, that is, there is at least one type I extension in the group representation.
Then
$$
\mathbf U\simeq\left(\PLPII{\PLPII{\ldots\left(\PLPII{\left(\PLPI{\mathbf X_{j-1}}{\mathbf H_{j-1}}{\mathbf G_j}\right)}{\mathbf G_{j+1}}\right)}{\ldots}}{\mathbf G_n}\right).
$$
By Lemma~\ref{AB-C=A-BC} it is isomorphic to
$$
\PLPII{\left(\PLPI{\mathbf X_{j-1}}{\mathbf H_{j-1}}{\mathbf G_j}\right)}{\left(\PLPII{\mathbf G_{j+1}}{\PLPII{\ldots}{\mathbf G_n}}\right)} ,
$$
and by Lemma~\ref{Vegyes} it is isomorphic to 
$$
\PLPI{\mathbf X_{j-1}}{\mathbf H_{j-1}}{\left(\PLPII{\mathbf G_j}{\PLPII{\ldots}{\mathbf G_n}}\right)}
$$
Applying Lemma~\ref{topbotJOLESZ} it follows that 
$\mathbf X_{j-1}\simeq_o\mathbb R$ and $\PLPII{\mathbf G_j}{\PLPII{\ldots}{\mathbf G_n}}\simeq_o\mathbb R$, and $\mathbf H_{j-1}$ is countable.
Thus, by Lemma~\ref{ZZZR} it follows that $\PLPII{\mathbf G_j}{\PLPII{\ldots}{\mathbf G_n}}\simeq\mathbb U_{n-j}$, and the induction hypothesis applied to $\mathbf X_{j-1}$ ends the proof.
\qed\end{proof}



\bibliography{easychair}
\end{document}